\documentclass[12pt, reqno]{amsart}
\usepackage{amssymb,amsmath}
\usepackage{amsmath, amsthm, amscd, amsfonts, amssymb, graphicx, color}
\usepackage{amsmath, amsthm, amssymb}
\usepackage{amsfonts}
\usepackage{graphicx}

\newtheorem{theorem}{Theorem}[section]
\newtheorem{lemma}{Lemma}[section]
\newtheorem{corollary}{Corollary}[section]

\newtheorem{remark}{Remark}[section]

\newtheorem{example}{Example}[section]
\newcommand{\M}{\mathbb{M}_{n}}
\newcommand{\MMM}{\mathbb{M}^s_{n, \alpha}}

\newcommand{\C}{\mathbb C}

\begin{document}

\title[Numerical Radius of Products of Special Matrices]{Numerical Radius of Products of Special Matrices}

\author[M. Alakhrass]{Mohammad Alakhrass}	

\address{ Department of Mathematics, University of Sharjah , Sharjah 27272, UAE.}
\email{\textcolor[rgb]{0.00,0.00,0.84}{malakhrass@sharjah.ac.ae}}

\subjclass[2010]{15A45, 15A60, 47A12, 47A30.}

\keywords{Numerical radius, Sectorial matrices, accretive-dissipative matrices, Matrix inequalities}	
\maketitle

\begin{abstract}
The purpose of this note is to present upper bounds estimations for the numerical radius of a products and Hadamard products of special matrices, including sectorial and accretive-dissipative matrices.
\end{abstract}

\section{Introduction}

Let $\M$ be the algebra of all $n \times n$ complex matrices. If $X=[x_{i,j}], Y=[y_{i,j}] \in \M$,
then their Hadamard product $X\circ Y$ is the matrix $[x_{i,j}y_{i,j}].$
The cartesian decomposition of $X \in \M$ is presented as
\begin{equation}\label{cartesian decomposition}
X=A+iB,
\end{equation}
where $A$ and $B$ are the Hermitian matrices $A= Re(X)=\frac{X+X^*}{2}$ and $B= Im(X)=\frac{X-X^*}{2i}$.
A matrix $X$ is said to be accretive (resp. dissipative) if in its cartesian decomposition \eqref{cartesian decomposition} the matrix $A$ (resp. $B$) is positive definite. If both $A$ and $B$, in the decomposition \eqref{cartesian decomposition}, are positive definite, $X$ is called accretive-dissipative.

The numerical range of $X \in \M$ is the compact convex subset of the complex plane defined as follows:
$$
W(X)=\{ \langle {Xx,x} \rangle: x \in \mathbb{C}^n, ||x||=1 \},
$$
where $\langle {\cdot, \cdot} \rangle$  is the standard inner product on $\C^n$ and $|| \cdot ||$ is the Euclidean norm on $\C^n$.
A very important result is that
$$
\sigma(X) \subset W(X),
$$
where $\sigma(X)$ is the spectrum of $X$.

For $\alpha \in [0, \pi/2)$, let $S_{\alpha}$ be the sector defined in the complex plane by
$$
S_{\alpha}=\{z \in \mathbb{C}: Re(z) > 0, |Im(z)| \leq \tan(\alpha) Re(z)  \}.
$$
A matrix $X$ is called sectorial if $W(z X) \subset S_{\alpha}$ for some complex number $z$ with $|z|=1$.
The smallest possible such $\alpha$ is called the index of sectoriality.

For $\alpha \in [0, \pi/2)$, let $\MMM$ be the class of all $n \times n$ matrices $X$ with $W(zX) \subset S_{\alpha}$ for
some complex number $z$ with $|z|=1$.

It is clear that $X$ is accretive-dissipative if and only if
$W (e^{-\pi i/4} X) \subset S_{\pi/4}$, and hence $X \in \MMM$ with $\alpha=\frac{\pi}{4}$. For more study of sectorial matrices see \cite{Alakhrass-2021, Alakhrass-Lieb, Alakhrass-A note in sectorial matrices, Yu-Popov-2003, Drury2013-4, Minghua2016, Minghua2017, Zhang-2015} and the references
therein.

A norm $N$ on $\mathbb{M}_n$ is said to be unitarily invariant if it satisfies the property $N(UXV)$ for all $X\in\mathbb{M}_n$ and all unitaries $U,V\in\mathbb{M}_n,$ and it is said to be multiplicative if $N(AB) \leq N(A) N(B)$ for all  $A,B \in \M$. Examples of
such unitarily invariant multiplicative norms are the Schatten $p$-norm defined by
$||X||_p = \left(\sum_{j=1}^{n} s^p_j(X) \right)^p,\;p\geq 1.$ When $p=\infty$, this last norm is just the usual operator norm defined by $\|X\|=\sup_{\|x\|=1}\|Xx\|.$

Associated with numerical range, the numerical radius of $X$ is defined by
$$
\omega(X)=\sup  \{ |z|: z \in W(X)\}.
$$
It is well known that $w(\cdot)$ defines a norm on $\M$  which is equivalent to the usual operator norm $\| \cdot \|$. In fact we have
\begin{equation}\label{A}
\frac{1}{2} \|X\|  \leq \omega(X) \leq   \| X \|; \quad  \forall X \in \M.
\end{equation}
Moreover, if $X \in\mathbb{M}_n$ is normal then $\omega(X)=\| X \|.$ Therefore, the inequalities in \eqref{A} are sharp.

Obviously, $\omega( \cdot )$ defines a weakly unitarily invariant norm on $\M$; that is
it satisfies the property
$\omega(UXU)=\omega(X)$ for all $X\in\mathbb{M}_n$ and all unitary $U \in\mathbb{M}_n.$

\begin{example}\label{E1}
Let
$
X=
\left(
  \begin{array}{cc}
    0 & 2 \\
    0 & 0 \\
  \end{array}
\right)
$
and
$
Y=
\left(
  \begin{array}{cc}
    0 & 0 \\
    2 & 0 \\
  \end{array}
\right).
$
Then
$\omega(X)= \omega(Y)= 1$ and $\omega(XY)= 4 $
\end{example}

Example \ref{E1} shows that $\omega(\cdot)$ is not a multiplicative norm. However, the inequalities in \eqref{A} implies that $4 \omega( \cdot )$
is a multiplicative norm. That is for all $X, Y \in \M$,
$4 \omega(XY) \leq \left(4 \omega(X) \right) \left( 4 \omega(Y)\right)$. Equivalently,

\begin{equation}\label{B}
\omega(XY) \leq  4 \omega(X) \omega(Y); \quad  \forall X, Y \in \M.
\end{equation}

Obviously, Example \ref{E1} shows that the inequality \eqref{B} is sharp and the constant 4
is the best possible in \eqref{B}.

The Hadamard product version of \eqref{B} can be written as

\begin{equation}\label{C}
\omega(X \circ Y) \leq 2 \omega(X) \omega(Y).
\end{equation}
That is, $2 \omega( \cdot )$ is a multiplicative norm over the Hadamard product. See \cite[p. 73]{Horn and Jonson-Topics-book-1991}.

The following example shows that the constant 2 is the best possible in \eqref{C}.

\begin{example}\label{E2}
Let
$
X=Y=
\left(
  \begin{array}{cc}
    0 & 2 \\
    0 & 0 \\
  \end{array}
\right)
$
Then
$\omega(X)= \omega(Y)= 1$ and $\omega(X \circ Y)= 2 $.
\end{example}

By considering special matrices $X$ and $Y$, it is possible to obtain better estimations than those in \eqref{B} and \eqref{C}.

If $XY=YX$, then
\begin{equation}\label{D}
\omega(XY) \leq 2 \omega(X) \omega(Y).
\end{equation}
See \cite[Theorem 2.5-2]{Gustafson-Rao-1997}.

If $X$ or $Y$ is normal such that $XY=YX$, then
\begin{equation}\label{E}
\text{} \quad \omega(XY) \leq \omega(X) \omega(Y).
\end{equation}
See \cite[Corollary 2.5-6]{Gustafson-Rao-1997}.

If $X$ or $Y$ is normal. Then

\begin{equation}\label{H}
\omega(X \circ Y) \leq \omega(X)\omega(Y).
\end{equation}
See \cite[Corollary 4.2.17]{Horn and Jonson-Topics-book-1991}.

If $A, B \in \M$ and $A=[a_{ij}]$ is positive semidefinite, then
\begin{equation}\label{I}
\omega(A \circ B) \leq   \left( \max_j a_{jj} \right) \, \, \omega(B).
\end{equation}
See \cite[Corollary 4]{Ando-1991} and \cite[Proposition 4.1]{Gau-Wu-2016}.

The purpose of this short note is to add more inequalities to the above list. More precisely, we give estimations of the numerical radius of
products or Hadamard products of sectorial matrices and related matrices such as accretive and dissipative matrices.

\section{Main Results}

We start this section by the following two observations.

\begin{lemma}\label{L1}
Let $X \in \M$. If $W(X) \subset S_{\alpha}$, then
$$
\| X \| \leq  \sec(\alpha) \| Re (X) \|.
$$
\end{lemma}

The above Lemma can be found in \cite{Alakhrass-2021}, \cite{Alakhrass-A note in sectorial matrices} , \cite{Alakhrass-Lieb} and \cite{Zhang-2015}.

\begin{remark} \label{R1}

\begin{enumerate}
\item
We recall that $\omega( \cdot )$ defines a self-adjoint norm on $\M$, that is
it is a norm  satisfies the properties $w(X^*)=w(X)$ for all $X \in \M$. Therefore,
\begin{align}
\omega \left(Re(X) \right) &= \omega \left(\frac{X + X^*}{2} \right)  \notag \\
&\leq  \frac{1}{2} \left( \omega \left( X+ X^* \right) \right) \notag \\
&\leq  \frac{1}{2} \left( \omega \left( X \right)  + \omega \left(X^* \right)  \right) \notag \\
&\leq  \frac{1}{2} \left( \omega \left( X \right)  + \omega \left(X \right)  \right) \notag \\
&\leq  \omega \left( X \right).  \notag
\end{align}

\item
Let $X \in \M$ and let $X=A + i B$ be its cartesian decomposition. If $W(X) \subset S_{\alpha}$, then for any $v \in \mathbb{C}^n$ with $||v||=1$
$ \langle Av \; , \; v \rangle, \langle Bx \; , \; x \rangle \in \mathbb{R}$ and
$$
\langle Xv \; , \; v \rangle =\langle Av \; , \; v \rangle + i \;  \langle Bv \; , \; v \rangle \in S_{\alpha}.
$$
Therefore,
$$
| \langle Bv \; , \; v \rangle | \leq   \tan(\alpha) \langle Av \; , \; v \rangle.
$$

Taking the supremum over all such v's gives

$$
\omega(B) \leq \tan(\alpha) \omega(A).
$$
\end{enumerate}
\end{remark}

Now we are ready to state the first result.

\begin{theorem}\label{Th-1}
Let $X\in \mathbb{M}^s_{n, \alpha_1}$ and $Y \in \mathbb{M}^s_{n, \alpha_2}$, where $\alpha_2, \alpha_2 \in [0, \pi/2)$. Then
$$
\omega(XY) \leq \sec(\alpha_1)\sec(\alpha_2) \omega(X)\omega(Y).
$$
\end{theorem}

\begin{proof}
Since $X \in \mathbb{M}^s_{n, \alpha_1}$ and $Y \in \mathbb{M}^s_{n, \alpha_2}$, there are two complex numbers
$z, w \in \mathbb{C}$ with $|z|=|w|=1$ such that $W(z X) \subset S_{\alpha_1}$ and
$W(w Y) \subset S_{\alpha_2}$.

Notice that
\begin{align}
\omega(XY)  & \leq  || XY|| \quad   \text{( by \eqref{A} ))} \notag \\
& \leq || X|| \, \, ||Y|| \notag  \\
& =|| z X|| \, \, || w Y||  \notag  \\
& \leq sec(\alpha_1)\sec(\alpha_2)  || Re \left(z X \right)| \, \,  | || Re \left( w  Y \right)||  \quad   \text{(by Lemma \ref{L1})}  \notag \\
& = sec(\alpha_1)\sec(\alpha_2)  \omega( Re \left(z X \right)) \, \,   \omega \left(Re \left( w  Y \right) \right) \notag \\
& \text{(since $Re \left(z X \right))$ and $ Re \left( w  Y \right)$ are Hermitian ) }  \notag \\
& \leq sec(\alpha_1)\sec(\alpha_2) \omega\left(z X  \right)  \, \, \omega(w  Y)  \quad   \text{(by Remark \ref{R1}) }  \notag \\
& = sec(\alpha_1)\sec(\alpha_2) \omega\left(X  \right)  \, \, \omega(Y).   \notag
\end{align}
\end{proof}

We remark that if $X,Y \in \MMM$, Theorem \ref{Th-1} implies that
\begin{equation}\label{B2}
\omega(XY) \leq \sec^2(\alpha)\omega(X)\omega(Y).
\end{equation}
The inequality \eqref{B2} presents a refinement of the inequality \eqref{B} when
$0 \leq \alpha \leq \frac{\pi}{3}$. A particular case is when $X$ and $Y$ are accretive-dissipative as in the following result.

\begin{corollary}
If $X, Y \in \M$ are accretive-dissipative, then
$$
\omega(XY) \leq 2 \omega(X)\omega(Y).
$$
\end{corollary}

\begin{proof}
The result follows from Theorem \ref{Th-1} and the fact that if $X, Y \in \M$ are accretive-dissipative then $X, Y \in \MMM$ with $\alpha = \pi /4.$
\end{proof}

\begin{remark}
\begin{enumerate}
\item If $X \in \M$ is accretive, i.e. $Re(X)>0$, then $X$ is sectorial with sectorial index
\begin{equation}\label{Accretive}
\alpha_X= \tan^{-1}(|\lambda_1 \left(   \left(Re X \right)^{-1} Im(X))|\right) =
|| \left(Re X \right)^{-1/2} \left( Im X \right)  \left(Re X \right)^{-1/2} ||.
\end{equation}
See \cite{Yu-Popov-2003}.
Therefor $X \in \MMM$ with $\alpha= \alpha_X$.

\item
If $X \in \M$ is dissipative, i.e. $Im(X)>0$, then $i X$ is accretive and hence  $X \in \MMM$ with $\alpha= \alpha_X$.

\end{enumerate}

\end{remark}

\begin{corollary}
If $X, Y \in \M$ are accretive (or dissipative), then
$$
\omega(XY) \leq  (1+a^2)  \omega(X)\omega(Y),
$$
where
$$
a = \max \{ | \lambda_1 \left ( \left(Re X \right)^{-1} Im X \right)| , |\lambda_1 \left (  \left(Re Y  \right)^{-1} Im Y \right)|\}.
$$
\end{corollary}

\begin{proof}
Since $X, Y \in \M$ are accretive (or dissipative), we have
$X \in \mathbb{M}^s_{n, \alpha_X}$ and $Y \in \mathbb{M}^s_{n, \alpha_Y}$,
where $\alpha_X$ and $\alpha_Y$ are as in \eqref{Accretive}. By Theorem \ref{Th-1}, we have
$$
\omega(XY) \leq  \sec(\alpha_X)\sec(\alpha_Y)  \omega(X)\omega(Y).
$$
Now the result follows by Noting that
$$
\sec(\alpha_X)=\sec(\tan^{-1}(|\lambda_1 \left(   \left(Re(X) \right)^{-1} Im(X))|\right))=
\sqrt{ 1 + \left( \lambda_1 \left(   \left(Re X \right)^{-1} Im X\right) \right)^2},
$$
and
$$
\sec(\alpha_Y)=\sec(\tan^{-1}(|\lambda_1 \left(   \left(Re Y \right)^{-1} Im Y)|\right))=
\sqrt{ 1 + \left( \lambda_1 \left(   \left(Re Y \right)^{-1} Im Y\right) \right)^2}.
$$

\end{proof}

Same proof method used to prove Theorem \ref{Th-1} can be used to prove the following more general result.

\begin{theorem}\label{Th-2}
Let $X_j\in \mathbb{M}^s_{n, \alpha_j}, j=1,2,3,...,m$. Then
$$
\omega\left( \prod_{j=1}^{m} X_j \right) \leq \prod_{j=1}^{m} \sec(\alpha_j) \omega(X_j).
$$
\end{theorem}

Consequently,

\begin{corollary}
If $X_1, X_2, ..., X_m \in \M$ are accretive-dissipative, then
$$
\omega\left( \prod_{j=1}^{m} X_j \right) \leq 2^{m/2} \prod_{j=1}^{m}  \omega(X_j).
$$

\end{corollary}

\begin{corollary}
If $X_1, X_2, ..., X_m \in \M$ are accretive (or dissipative), then
$$
\omega\left( \prod_{j=1}^{m} X_j \right)  \leq  (1+a^2)^{m/2} \prod_{j=1}^{m}  \omega(X_j),
$$
where
$$
a = \max \{ | \lambda_1 \left (   \left(Re X_j \right)^{-1} Im X_j  \right )| , j=1,2,...,m   \}.
$$
\end{corollary}

In what follows, we present upper bounds for the numerical ranges of hadamard products of special matrices. The following Lemma is important in our analysis.
It can be found in
\cite{Alakhrass-2021, Alakhrass-Lieb, Alakhrass-A note in sectorial matrices}.

\begin{lemma}\label{Characterization of T}
Let $T\in \M$. If $W(T) \subset S_{\alpha}$ for some $\alpha\in [0,\pi/2)$.
Then
$$
\left(
  \begin{array}{cc}
   \sec(\alpha) \Re(T) & T \\
    T^* & \sec(\alpha) \Re(T) \\
  \end{array}
\right)\geq 0.
$$
\end{lemma}

Now we estimate the numerical range for a Hadamard product of sectorial matrices.

\begin{theorem}\label{H1}
Let $X\in \mathbb{M}^s_{n, \alpha_1}$ and $Y \in \mathbb{M}^s_{n, \alpha_2}$. Then
$$
\omega(X \circ Y) \leq \sec(\alpha_1)\sec(\alpha_1)  \omega(X)\omega(Y).
$$

\end{theorem}

\begin{proof}
Since $X\in \mathbb{M}^s_{n, \alpha_1}$ and $Y \in \mathbb{M}^s_{n, \alpha_2}$, there are two complex numbers
$z, w \in \mathbb{C}$ with $|z|=|y|=1$ such that $W(z X) \subset S_{\alpha_1}$ and $W(w Y) \subset S_{\alpha_2}$.
Therefore, by Lemma \ref{Characterization of T}, the following two block matrices
$$
\left(
  \begin{array}{cc}
    \sec(\alpha_1)Re(zX) & zX \\
    \overline{z} X^* &  \sec(\alpha_1) Re(zX)\\
  \end{array}
\right),
\, \, \,
\left(
  \begin{array}{cc}
    \sec(\alpha_2) Re(wY) & wY \\
    \overline{w}Y^* &  \sec(\alpha_2) Re(wY) \\
  \end{array}
\right)
$$
are positive semidefinite. Hence
$$
\left(
  \begin{array}{cc}
    \sec(\alpha_1)\sec(\alpha_2) Re(zX) \circ Re(wY) &zw \left( X \circ Y \right) \\
    \overline{zw} \left (X \circ Y \right)^* &  \sec(\alpha_1)\sec(\alpha_2) Re(zX) \circ Re(wY) \\
  \end{array}
\right)
$$
is also positive semidefinite. Since $|| \cdot ||$ is a Lieb Function, we have
$$
||X \circ Y || =||zw \left( X \circ Y \right)|| \leq  \sec(\alpha_1)\sec(\alpha_2) || Re(zX) \circ Re(wY) ||.
$$
Now, observe that
\begin{align}
\omega (X \circ Y ) & \leq || X \circ Y  || \notag \\
& \leq  \sec(\alpha_1)\sec(\alpha_2)  || Re(zX) \circ Re(wY) ||   \notag \\
& = \sec(\alpha_1)\sec(\alpha_2)   \omega( Re(zX) \circ Re(wY))\notag \\
&\leq  \sec(\alpha_1)\sec(\alpha_2)  \omega( Re(zX)) \omega( Re(wY)) \quad (\text{by \eqref{H}})        \label{AA}  \\
&\leq  \sec(\alpha_1)\sec(\alpha_2)  \omega( zX ) \omega( wY ) \quad (\text{by Remark \ref{R1}})    \notag \\
&= \sec(\alpha_1)\sec(\alpha_2)  |z| \omega( X )  |w| \omega( Y ) \notag \\
&=  \sec(\alpha_1)\sec(\alpha_2)   \omega( X )  \omega( Y ). \notag
\end{align}
\end{proof}

\begin{remark}
\begin{enumerate}
\item
If $X,Y \in \MMM$, then
\begin{equation}\label{C2}
\omega(X \circ  Y) \leq \sec^2(\alpha)\omega(X)\omega(Y).
\end{equation}
It is clear that inequality \eqref{C2} presents a refinement of inequality \eqref{C} for $0 \leq \alpha < \frac{\pi}{4}$.
In particular if $X$ and $Y$ are accretive-dissipative, then both
\eqref{C2} and \eqref{C} give the same estimation.

\item
If $X, Y \in \M$ are accretive (or dissipative), then \eqref{C2} implies that
$$
\omega(X \circ Y) \leq  (1+a^2)  \omega(X)\omega(Y),
$$
where
$$
a = \max \{ | \lambda_1 \left ( \left(Re X \right)^{-1} Im X \right)| , |\lambda_1 \left (  \left(Re Y  \right)^{-1} Im Y \right)|    \}.
$$
\end{enumerate}
\end{remark}

The argument used to prove Theorem \ref{H1} can be easily modified to prove the following more general result.

\begin{theorem}
Let $X_j \in \mathbb{M}^s_{n, \alpha_1}, j=1,2,...,m$. Then
$$
\omega(X_1 \circ ... \circ X_m ) \leq \prod_{j=1}^{m} \sec(\alpha_j) \omega(X_j).
$$
\end{theorem}

Consequently,

\begin{corollary}
If $X_1, X_2, ..., X_m \in \M$ are accretive-dissipative, then
$$
\omega\left( X_1 \circ ... \circ X_m  \right) \leq 2^{m/2} \prod_{j=1}^{m}  \omega(X_j).
$$
\end{corollary}

\begin{corollary}
If $X_1, X_2, ..., X_m \in \M$ are accretive (or dissipative), then
$$
\omega\left( X_1 \circ ... \circ X_m  \right)  \leq  (1+a^2)^{m/2} \prod_{j=1}^{m}  \omega(X_j),
$$
where
$$
a = \max \{ | \lambda_1 \left (   \left(Re X_j \right)^{-1} Im X_j  \right )| , j=1,2,...,m   \}.
$$
\end{corollary}

In the following result we give an estimation for $\omega(X \circ Y)$ in terms of the diagonal entries of $X$ and $Y$.

\begin{theorem}
Let $X=[x_{ij}] \in \mathbb{M}^s_{n, \alpha_1}$ and $Y=[y_{ij}] \in \mathbb{M}^s_{n, \alpha_2}$. Then
$$
\omega(X \circ Y) \leq \sec(\alpha_1)\sec(\alpha_2)
\min \left \{ \max_{j} | x_{jj} |  \, \,  \omega(Y), \max_{j} | y_{jj}| \, \,  \omega(X) \right\}.
$$
\end{theorem}

\begin{proof}
Since $X\in \mathbb{M}^s_{n, \alpha_1}$ and $Y \in \mathbb{M}^s_{n, \alpha_2}$, the inequality in \eqref{AA} implies that
$$
\omega (X \circ Y ) \leq  sec(\alpha_1)\sec(\alpha_2) \omega( Re(zX) \circ Re(wY)),
$$
for some $z, w \in \mathbb{C}$ with $|z|=|w|=1$.
Since $Re(zX)$ is positive semidefinite, \eqref{I} implies that
$$
\omega( Re(zX) \circ Re(wY)) \leq  \max_j \, Re(zx_{jj}) \, \omega \left(  Re(wY) \right).
$$
Now, we have

\begin{align}
\omega (X \circ Y ) & = sec(\alpha_1)\sec(\alpha_2) \omega( Re(zX) \circ Re(wY))  \notag \\
&\leq  sec(\alpha_1)\sec(\alpha_2)  \max_j \, Re(zx_{jj}) \, \omega \left(  Re(wY) \right) \notag \\
&\leq  sec(\alpha_1)\sec(\alpha_2)  \max_j \, |z x_{jj}| \, \omega \left(  wY \right) \notag \\
& =    sec(\alpha_1)\sec(\alpha_2)  \max_j \, |x_{jj}| \, \omega \left(  Y \right). \notag
\end{align}

Hence,
\begin{equation}\label{N1}
\omega (X \circ Y ) \leq sec(\alpha_1)\sec(\alpha_2)   \max_j \, |x_{jj}| \, \omega \left(  Y \right).
\end{equation}

Similarly, we have

\begin{equation}\label{N2}
\omega (X \circ Y ) \leq sec(\alpha_1)\sec(\alpha_2)   \max_j \, |y_{jj}| \, \omega \left(  X \right).
\end{equation}
The result follows by combining \eqref{N1} and \eqref{N2}.
\end{proof}

\begin{corollary}
If $X=[x_{ij}], Y=[y_{ij}] \in \M$ are accretive-dissipative, then
$$
\omega(X \circ Y) \leq 2 \min \{ \max_{j} | x_{jj} |  \, \,  \omega(Y), \max_{j} | y_{jj}| \, \,  \omega(X) \}.
$$
\end{corollary}

\begin{corollary}
If $X, Y \in \M$ are accretive (or dissipative), then
$$
\omega(X \circ Y) \leq  (1+a^2) \min \left \{ \max_{j} | x_{jj} |  \, \,  \omega(Y), \max_{j} | y_{jj}| \, \,  \omega(X) \right\},
$$
where
$$
a = \max \{ | \lambda_1 \left (   \left(Re X \right)^{-1} Im X  \right )| , | \lambda_1 \left (   \left(Re Y \right)^{-1} Im Y  \right )| \}.
$$
\end{corollary}

Another upper bound for the numerical radius of a Hadamard product of two sectorial matrices can be obtained as follows.

\begin{theorem}
Let $X \in \mathbb{M}^s_{n, \alpha_1}$ and $Y \in \mathbb{M}^s_{n, \alpha_2}$. Then
$$
\omega(X \circ Y)
\leq \min  \left \{(1+ \tan \alpha_1)  \omega (Re \, X ) \omega (Y), (1+ \tan \alpha_2)\omega (X)  \omega (Re Y ) \right\}.
$$
Consequently,
$$
\omega(X \circ Y) \leq (1+ \tan \alpha)\omega (X)  \omega (Y ),
$$
where $\alpha= \max \{ \alpha_1, \alpha_1\}$.
\end{theorem}

\begin{proof}
The second inequality follows from the the first one and fact that $\omega(Re X) \leq \omega(X) \, \forall X \in \M$.
To prove the first inequality, let $X=A + i B$ be the cartesian decomposition of $X$.
Then
\begin{align}\label{zzz}
\omega(X \circ Y) & = \omega((A + i B) \circ Y)    \notag \\
& = \omega(A \circ Y  + i B \circ Y) \notag \\
& \leq \omega(A \circ Y) + \omega (B \circ Y) \notag \\
& \leq \omega(A) \omega (Y) + \omega (B) \omega (Y) \notag \quad \text{by \eqref{H}}  \notag \\
& \leq \omega(A) \omega (Y) + \tan(\alpha_1) \,  \omega (A) \omega (Y) \notag \quad \text{(by part 2 of Remark \ref{R1})}  \notag \\
& \leq (1+ \tan \alpha_1)\omega (A) \omega (Y) \notag \\
& =    (1+ \tan \alpha_1)\omega (Re(X) ) \omega (Y). \notag
\end{align}

Hence,
\begin{equation}\label{F5}
\omega(X \circ Y) \leq (1+ \tan \alpha_1)  \omega (Re X ) \omega (Y).
\end{equation}

Similarly, one can show that
\begin{equation}\label{F6}
\omega(X \circ Y) \leq (1+ \tan \alpha_2)  \omega (Re Y ) \omega (X).
\end{equation}

The result follows by combining \eqref{F5} and \eqref{F6}.
\end{proof}

\begin{corollary}
Let $X, Y \in \MMM$. Then
$$
\omega(X \circ Y)
\leq (1+ \tan \alpha)   \min  \left \{(\omega (Re \, X ) \omega (Y), \omega (X)  \omega (Re Y ) \right\}.
$$
Consequently,
\begin{equation}\label{Z}
\omega(X \circ Y) \leq (1+ \tan \alpha)  \omega (X) \omega (Y).
\end{equation}
\end{corollary}

Finally, we remark that inequality \eqref{Z} presents an improvement to \eqref{C} for $0 \leq \alpha  < \frac{\pi}{4}$.

\subsection*{Conflict of interest}
The authors declare that they have no conflict of interest.

\end{document}